\documentclass[review]{elsarticle}

\usepackage{amssymb,amsthm}
\usepackage{graphicx}
\usepackage{algorithm,algorithmic}	
\usepackage{amsmath,todonotes}
\usepackage{epstopdf}

\newtheorem{problem}{Problem}

\newtheorem{lemma}{Lemma}
\newtheorem{remark}{Remark}

\newtheorem{assumption}{Assumption}
\newtheorem{theorem}{Theorem}

\usepackage[left=2.5cm,right=2.5cm]{geometry}

\renewcommand{\vec}{\mathrm{vec}}
\newcommand{\rank}{\mathrm{rank}}

\newcommand{\ff}[1]{{\color{black}#1}}
\usepackage{bbm}

\begin{document}

\begin{frontmatter}

\title{On Reconstructability of Quadratic Utility Functions from the Iterations in Gradient Methods\tnoteref{t1}} 

\tnotetext[t1]{The work of I.~Shames is supported by McKenzie Fellowship. The work of F.~Farokhi is supported by the Australian Research Council (LP130100605). }

\author[add1]{Farhad~Farokhi\corref{cor}}
    \ead{farhad.farokhi@unimelb.edu.au}
    
\author[add1]{Iman~Shames}
    \ead{iman.shames@unimelb.edu.au}
    
\author[add2]{Michael~G.~Rabbat}
	\ead{michael.rabbat@mcgill.ca}
	
\author[add3]{Mikael Johansson}
	\ead{mikaelj@ee.kth.se}
  
\cortext[cor]{Corresponding author}

\address[add1]{Department of Electrical and Electronic Engineering,  University of Melbourne, Parkville, Australia}

\address[add2]{Department of Electrical and Computer Engineering, 
McGill University, Montr\'{e}al, Qu\'{e}bec, Canada}

\address[add3]{ACCESS Linnaeus Center, Electrical Engineering, KTH Royal Institute of Technology,
Stockholm, Sweden}

\begin{abstract}
In this paper, we consider a scenario where an \emph{eavesdropper} can read the content of messages transmitted over a network. The nodes in the network are running a gradient algorithm to optimize a quadratic utility function where such a utility optimization is a part of a decision making process by an \emph{administrator}. We are interested in understanding the conditions under which the eavesdropper can reconstruct the utility function or a scaled version of it and, as a result, gain insight into the decision-making process. We establish that if the parameter of the gradient algorithm, i.e.,~the step size, is chosen appropriately, the task of reconstruction becomes practically impossible for a class of Bayesian filters with uniform priors. We establish what step-size rules should be employed to ensure this.

\end{abstract}


\end{frontmatter}

\section{Introduction}
In recent decades, tremendous advances in the areas of communication and computation have facilitated the construction of complex systems. The design and analysis of these systems involve solving large optimization problems.
Utility maximization, optimal flow, expenditure minimization, and traffic optimization are examples of such problems. Due to the size of these problems, it is often required  that problems are solved over a network of interconnected processors. In many scenarios, the implementation of the solution to the optimization problem is in the public domain. However, from an operational point of view, it is important that the way that the decision is made remains confidential. In other words, while the optimal decision can be known by everyone, the utility function itself should remain confidential. Portfolios in portfolio optimization and local utilities in resource allocation can be considered as examples of such utility functions that need to be kept confidential. This especially becomes an issue as more computations related to operating the critical infrastructure (e.g.~power distribution networks) are carried out in the cloud~\cite{akyol2012cyber}. The importance of confidentiality, integrity, and availability are well understood in the security of data and ICT services~\cite{bishop2002computer} and cloud computing~\cite{chen2010s}. In these settings, confidentiality corresponds to ensuring the non-disclosure of data, integrity is related to the trustworthiness of data,  and availability is concerned with the timely access to the data or system functionalities.

In this paper, we mainly focus on the question of confidentiality---particularly, the confidentiality of the utility functions even when the security of the network is compromised and an eavesdropper can listen to all the information being exchanged over the network during the course of solving the optimization problem. We consider scenarios where the utility function has a quadratic form. Specifically, the following question is answered: \emph{when is it possible to reconstruct a utility function, or a scaled version of it, via having access to the iterations produced by an iterative method?}  The iterative method considered in this paper is a gradient ascent algorithm. The choice of a gradient algorithm is inline with the recent observations that cast a favourable light on employing first-order methods to solve very large optimization problems \cite{cevher2014convex}. Note that
the choice of quadratic programs is not very restrictive as  trust-region optimization techniques allow us to solve any general optimization problem using a sequence of constrained quadratic programs recursively~\cite{wright1999numerical}.

The problem that is addressed here is related to the one considered in the context of differential privacy \cite{dwork2008differential} and, to a larger extent, the application of differential privacy in optimization~\cite{Gupta2010,Chaudhuri2011,Mangasarian2011}. However, it is important to note that there, the price for guaranteed confidentiality is paid in terms of data integrity and the accuracy of the solution.  To ensure differential privacy, it is known that the information passed between the processing nodes at each step of the optimization algorithm should be perturbed by a random variable from a Laplace distribution \cite{dwork2008differential}. This results in the algorithm not yielding an accurate  solution. Here, we argue that the confidentiality of the objective (but not the solution) can be guaranteed in practice with no impact on the accuracy of the solution, if the algorithm parameter (the step size) is chosen appropriately, i.e.,~it is picked randomly from a sufficiently large set of suitable step sizes. In addition to differential privacy, other notions of privacy in optimization and machine learning have recently been pursued, e.g., see~\cite{weeraddanaper,Vaidya2008,NIPS2012_4505}. Note that, in this paper, we are not directly contributing to the privacy-preserving literature, \ff{\textit{per se}}. Our main objective is point out that, in the setups discussed, one does not need to worry about privacy since estimating the underlying parameters is practically impossible due computational restrictions (at least, with current technologies). Note that this problem is also related to the system identification and the parameter estimation literature, where the aim is to extract the parameters of the underlying utility or dynamics. However, in our setup, the eavesdropper cannot inject proper reference signals to fully probe the system (that is commonly known as the \textit{persistent excitation} and is necessary for achieving the estimation objective~\cite{1103142}). Finally, in~\cite{gentry2001dynamic}, the agents use their actions to learn about the strategies or the utilities of the other agents to subsequently devise optimal strategies. However, in that study, the computational aspects of the problem were largely unexplored and only linear programs were considered.

The outline of this paper is as follows. In the next section, the problem that is considered in this paper is formulated. In Section \ref{sec:qp_unc}, we consider the case where the eavesdropper has access to the iterates that are generated during the course of solving an unconstrained quadratic program. In this section, different choices of the step size are considered and conditions for which the utility function cannot be constructed successfully are discussed. Next, in Section \ref{sec:qp_c}, we consider the case where the problem is constrained. Concluding remarks are given in Section \ref{sec:conc}.

\subsection{Notation} The sets of reals, nonnegative reals, integers, and nonnegative integers are, respectively, denoted by $\mathbb{R}$, $\mathbb{R}_{\geq 0}$, $\mathbb{Z}$, and $\mathbb{Z}_{\geq 0}$. The rest of the sets are denoted by calligraphic Roman letters, such as $\mathcal{M}$. Specifically, $\mathcal{S}_+^n$ is defined to be the set of symmetric positive-definite matrices in $\mathbb{R}^{n\times n}$. We define $\vec:\mathbb{R}^{n\times m}\rightarrow \mathbb{R}^{nm}$ to be a vectorization operator that puts all the columns of a matrix into a vector sequentially. Finally, we use $A\otimes B$ to be the Kronecker product of matrices $A$ and $B$.

\section{Problem Formulation}\label{sec:prob_form}
Consider the following optimization problem:
\begin{subequations} \label{eq:opt}
\begin{align}
\max_{x\in\mathbb{R}^n}\quad & -\dfrac{1}{2} x^\top Q x - q^\top x,\\
\text{s.t.} \quad & \ff{Cx \leq d},
\end{align}
\end{subequations}
where $Q\in \mathcal{S}_+^{n}$, $q\in \mathbb{R}^n$, $\ff{C}\in \mathbb{R}^{m\times n}$, $\ff{d}\in\mathbb{R}^m$, and $\mathcal{X}\triangleq\{x\in\mathbb{R}^n\,|\,\ff{Cx<d}\}\neq \emptyset$. The optimization problem \eqref{eq:opt} is solved by an \emph{administrator} over a network via an optimization method, $\mathcal{F}(\cdot)$, given by 
\begin{equation} \label{eqn:algorithmdef}
x[k+1]=\mathcal{F}(x[k]),\,x[0]\in\mathcal{X}.
\end{equation}
Throughout this paper, we assume that $\mathcal{F}(\cdot)$ is the gradient ascent algorithm in which different step-size selection methods can be used. This assumption is, partly, motivated by favourable results on first-order methods for solving large-scale optimization problems \cite{cevher2014convex}. However, this assumption is also in place to greatly simplify the proofs and the presentation.

\begin{remark} At first glance, the update rule in~\eqref{eqn:algorithmdef} appears to be a centralized implementation. However, distributed algorithms using primal decomposition as well as the inner problems for distributed algorithms using dual decomposition (see~\cite{bertsekas1997parallel}) can be both rewritten, albeit in an aggregated form, in the form of~\eqref{eqn:algorithmdef}. 
\end{remark}

\begin{remark} The results presented in this work, at least in part, are applicable to more general utility functions, e.g., logarithmic functions. However, the selection of the quadratic utility functions results in linear operators that greatly simplifies the proofs. Moreover, the quadratic utility functions, although partially conservative, have many applications and are widely used in signal processing, e.g., weighted least squares, and machine learning, e.g. support vector machines (SVM) \cite{Vaidya2008}. 
\end{remark}

The measurement model of the eavesdropper is as follows. For any two consecutive measurements of the optimization variable $x[k]$ and $x[k+1]$, for some $k\in\mathbb{Z}_{\geq 0}$, the eavesdropper can construct a measurement of the form
\begin{align}\label{eq:eaves_model}
y[k]=\ff{x[k]-x[k+1]}.
\end{align}
Therefore, at time step $k+1$, the eavesdropper has access to measurement pairs $(x[t],y[t])_{t=0}^k$.  Providing the solution to the following problems is of interest. 
\begin{problem}[Utility Function Reconstruction]\label{prob:main}
Assuming that the eavesdropper can measure $x[k]$ for all $k$ and the values of $A$ and $b$ are known, under what conditions on the step size selection of the gradient descent algorithms can the eavesdropper estimate $(\hat{Q},\hat{q})$ such  that $Q = \gamma \hat{Q}$ and $q = \gamma\hat{q}$ for some $\gamma>0$?
\end{problem}
Solving the problem above enables the eavesdropper to determine the way that decisions are made. For example, it can be determined which variable has a bigger impact on the solution of the optimization problem~\eqref{eq:opt}. Hence, it is not necessary to exactly estimate $\gamma$.
\begin{remark}
In this paper, we assume that the communication is carried out over real and noiseless channels. Alternatively, one may consider the effects of quantisation and noise on the utility reconstruction problem. However, this is beyond the scope of this paper.
\end{remark}
Finally, we have the following standing assumption.
\begin{assumption} \label{assum:Qdistribution} The parameters $(Q,q)\in\mathcal{Q}\subseteq \mathcal{S}_{+}^n\times\mathbb{R}^n$ are randomly generated according to the non-degenerate probability density function $p:\mathcal{Q}\rightarrow \mathbb{R}_{\geq 0}$. Further, we assume the distribution of $(Q,q)$ is independent of the initialization of the algorithm $x[0]$, which is uniformly selected from $\{x|x^\top x\leq 1\}$. The eavesdropper knows these probability distributions.
\end{assumption}
In the following sections, first we consider the case where the problem \eqref{eq:opt} is unconstrained, and then we study the constrained case. Note that the choice of unconstrained quadratic programs is not restrictive. We can solve any general optimization problem using a sequence of constrained quadratic programs recursively using trust-region optimization techniques. Further, when using primal-dual techniques, we can solve a constrained quadratic program recursively using a sequence of unconstrained quadratic programs. Alternatively, as also discussed in Section~\ref{sec:qp_c}, we can use logarithmic barrier functions to solve a constrained quadratic program.
 
\section{Unconstrained Case} \label{sec:qp_unc}
In the case where the optimization problem is unconstrained, the gradient iterations are such that
\begin{equation} \label{eqn:algorithmdef:gradient}
x[k+1]=x[k]\ff{-}\alpha[k]  ( Q x[k] +q ) ,
\end{equation}
where $\alpha[k]$ is an appropriately selected step size (e.g., it is well known that if $\alpha[k]$, $\forall k\in\mathbb{Z}_{\geq 0}$, belongs to an appropriately selected interval on the positive reals, the iterations in~\eqref{eqn:algorithmdef:gradient} converge to the optimal solution~\cite{goldstein1964convex}). 

\begin{remark} As remarked earlier, in this paper our primary interest is in the case where the gradient iterations are implemented in a distributed manner. For instance, if we employ $n$ processors, each processor needs to follow the update rule
\begin{align*}
x_i[k+1]=(1\ff{-}\alpha[k]q_{ii})x_i[k]\ff{-}\sum_{j\neq i} \alpha[k]q_{ij}x_j[k]\ff{-} \alpha[k]q_i,
\end{align*}
where $x_i[k]$ is $i$'th element of the decision vector $x[k]$. 
To implement this update rule, the processors need to communicate the elements of the decision vector over the directed graph $\mathcal{G}$ with vertex set $\mathcal{V}_\mathcal{G}=\{1,\dots,n\}$ and edge set $\mathcal{E}_\mathcal{G}=\{(i,j)|1\leq i\neq j\leq j,q_{ij}\neq 0\}$.  The messages that the processors pass contain $x_i[k]$, $\forall i$, and by observing these messages, the eavesdropper can obtain $(x[t])_{t \in\mathbb{Z}_{\geq 0}}$. As a viable avenue for future research, we can consider the scenario in which the eavesdropper can only listen to a subset of the transmitted messages.
\end{remark}

In this scenario, the eavesdropper measurement model given by \eqref{eq:eaves_model} becomes
$$
y[k]=\alpha[k](Qx[k]+q).
$$
As before, at time step $k+1$, measurement pairs $(x[t],y[t])_{t=0}^k$ are available to the eavesdropper. We make the following standing assumption.
\begin{assumption} \label{assum:independent} The vectors $(x[t])_{t=0}^{n-1}$ are linearly independent.
\end{assumption}
We formally characterize the conditions for which Assumption \ref{assum:independent} holds in the course of this paper for each scenario. 

\begin{remark} 
Under Assumption~\ref{assum:Qdistribution}, the independence assumption is without loss of generality (i.e., Assumption~\ref{assum:independent} holds almost surely). Note that when using the gradient ascent algorithm, for the measurements to be dependent, the algorithm should be initialized at {a point along one of the eigenvectors} of $Q$, and $q$ should be parallel to that specific {eigenvector}. 
\end{remark}

\begin{remark} \label{remark:explainNecessary} Let us briefly explain why Assumption~\ref{assum:independent} is necessary. \ff{In this remark, we assume that $\alpha[k]$, $k\in\mathbb{N}$, is known. Note that this shows the necessity of Assumption~\ref{assum:independent} as  it discusses a more restrictive setup (because the eavesdropper has access to more information). }
In such case, we have
\begin{align*}
Qx[k]+q
&=\vec(Qx[k]+q)\\
&=\vec(Qx[k])+q\\
&=(x[k]^\top\otimes I)\vec(Q)+q, 
\end{align*}
where the last equality follows from Item~(5) in~\cite[p.\,97]{lutkepohl1997handbook}. This gives
\begin{align} \label{eqn:set_of_linear_eq}
\left(
\begin{bmatrix}
x[0]^\top & 1 \\
\vdots & \vdots \\
x[k]^\top & 1
\end{bmatrix}
\otimes I
\right)
\begin{bmatrix}
\vec(Q) \\ q
\end{bmatrix}
=
\begin{bmatrix}
y[0]/\alpha[0] \\
\vdots \\
y[k]/\alpha[k] 
\end{bmatrix}.
\end{align}
Let us denote the matrix on the left hand side of~\eqref{eqn:set_of_linear_eq} by $G$. 
To avoid admitting redundant equations, $G$ should have a full row rank. From the properties of the Kronecker  product~\cite[p.\,58]{lutkepohl1997handbook}, if Assumption~\ref{assum:independent} holds, for all $k\leq n-1$, we know that 
\begin{align*}
\rank\left(
\begin{bmatrix}
x[0]^\top & 1 \\
\vdots & \vdots \\
x[k]^\top & 1
\end{bmatrix}
\otimes I
\right)=\rank\left(
\begin{bmatrix}
x[0]^\top & 1 \\
\vdots & \vdots \\
x[k]^\top & 1
\end{bmatrix}
\right)\rank(I)=(k+1)n.
\end{align*}
The number of rows of $G$ is also equal to $(k+1)n$ and, hence, $G$ has full row rank. Consequently, there is no redundant equation.
\end{remark}

\subsection{Constant Step Size}
In this subsection, we assume $\alpha[k]=\alpha>0$ for all $k\in\mathbb{Z}_{\geq 0}$. In this case, it is not possible to reconstruct $Q,q$ uniquely because $\alpha$ shows itself as a scaling factor in these matrices.  In other words, for $\gamma$ in Problem \ref{prob:main} we have $\gamma=1/\alpha$. Let us construct the set
$\mathcal{M}[k]=\big\{(Q',q')\in\mathcal{S}_+^n\times \mathbb{R}^n\,|\, y[t]=Q'x[t]+q', \forall t=0,\dots,k \big\}.$ 
 This denotes the set of parameters that are consistent with our measurements up to time step $k+1$ (note that for constructing $(y[t])_{t=0}^k$, the eavesdropper needs to measure $(x[t])_{t=0}^{k+1}$).  Evidently, by construction, these sets are nonexpansive, i.e., $\mathcal{M}[k+1]\subseteq \mathcal{M}[k]$ for any $k\in\mathbb{Z}_{\geq 0}$. First, let us present a condition under which Assumption~\ref{assum:independent} holds.
 
\begin{remark} \label{remark:smi} The introduced estimator $\mathcal{M}[k]$ has close connections to the idea behind set-membership identification in which the set of permissible parameters are gradually reduced by removing realisations that are not compatible with the newly received measurements; see~\cite{kosut1992set,Milanese1996,milanese2004set}. \end{remark}
 
\begin{lemma} \label{lemma:independent:1} Let the distribution $p$ governing $Q$ (cf.~Assumption~ \ref{assum:Qdistribution}) be such that the algebraic multiplicity of every eigenvalue of $Q$ is almost surely equal to one. Then, $(x[t])_{t=0}^{n-1}$ are almost surely independent.
\end{lemma} 
 
\begin{proof} Notice that
\begin{align*}
x[k+1]
&=x[k]\ff{-}\alpha(Qx[k]+q)=(I\ff{-}\alpha Q)x[k]\ff{-}\alpha q.
\end{align*}
Therefore, we have
\begin{align} \label{eqn:rankequality}
\begin{bmatrix}
x[0]^\top \\ x[1]^\top \\ \vdots \\ x[n-1]^\top
\end{bmatrix}
=
\begin{bmatrix}
x[0]^\top \\ x[0]^\top(I\ff{-}\alpha Q)^\top \\ \vdots \\ x[0]^\top(I\ff{-}\alpha Q)^{(n-1)\top}
\end{bmatrix}\ff{-}
\begin{bmatrix}
0_{1\times n} \\ (\alpha q)^\top \\ \vdots \\ \sum_{j=0}^{n-2}(\alpha q)^\top(I\ff{-}\alpha Q)^{j\top}
\end{bmatrix}.
\end{align}
Since $x[0]$ is selected randomly and independently from the pair $(Q,q)$, the iterates $(x[t])_{t=0}^{n-1}$ are independent (i.e., the matrix on the left-hand side of~\eqref{eqn:rankequality} is full rank) if 
\begin{align} \label{eqn:condition1}
\rank\left(
\begin{bmatrix}
x[0]^\top \\ x[0]^\top(I\ff{-}\alpha Q)^\top \\ \vdots \\ x[0]^\top(I\ff{-}\alpha Q)^{(n-1)\top}
\end{bmatrix}\right)=n,
\end{align}
which is equivalent to that the pair $((I\ff{-}\alpha Q)^\top,x[0]^\top)$ is observable. From the controllability/observability literature~\cite[p.\,123]{levine1999control}, we know that~\eqref{eqn:condition1} holds if and only if
\begin{align*}
\rank\left(
\begin{bmatrix}
(I\ff{-}\alpha Q)^\top-\mu I \\ x[0]^\top
\end{bmatrix}\right)=n
\end{align*}
for all eigenvalues $\mu$ of $(I\ff{-}\alpha Q)^\top$. This condition is satisfied if (\textit{i})~the algebraic multiplicity of all the eigenvalues of $(I\ff{-}\alpha Q)^\top$ is equal to one~\cite[p.\,59]{lutkepohl1997handbook} and (\textit{ii}) $x[0]$ is not an eigenvector of $(I\ff{-}\alpha Q)^\top$ (which is satisfied with probability one since $x[0]$ and $Q$ are drawn independently). This is, in turn, satisfied if the algebraic multiplicity of every eigenvalue of $Q$ is equal to one.
\end{proof}

The following theorem shows that after enough measurements, the set $\mathcal{M}[k]$ becomes a singleton.

\begin{theorem} \label{tho:constnat_step_size} Let $n\geq 3$. Then $\mathcal{M}[k]=\{(\alpha Q,\alpha q)\}$ for all $k\geq \lceil (n+1)/2 \rceil$.
\end{theorem}

\begin{proof}
Because $Q'$ is a symmetric matrix, it only has $(n^2+n)/2$ unknowns ($=1+2+\dots+n$). Therefore, the eavesdropper needs to calculate $(n^2+3n)/2$ unknowns corresponding with the entries of $Q'$ and $q'$. Due to Assumption~\ref{assum:independent}, if the eavesdropper collects measurements for up to $k=\lceil (n+1)/2 \rceil$ (notice that, by definition, the number of measurements in $\mathcal{M}[k]$ is equal to $k+1$), the set of linear equations defining $\mathcal{M}[k]$ admits a unique solution, $(\hat{Q},\hat{q})$, where $\hat{Q}=\alpha Q$ and $\hat{q}=\alpha q$. To be able to use Assumption~\ref{assum:independent}, we should have $\lceil (n+1)/2 \rceil=k\leq n-1$ based on Remark~\ref{remark:explainNecessary}, which gives $n\geq 3$. 
\end{proof}

\begin{remark}
If the eavesdropper collects  $k <  \lceil (n+1)/2 \rceil$ measurements, the set $\mathcal{M}[k]$ has infinitely many elements. As a result, if the iterations are terminated at $k <  \lceil (n+1)/2 \rceil$ iterations, it is impossible for the eavesdropper to reconstruct the parameters of the utility function. However, the confidentiality is guaranteed here at the price of getting a possibly inaccurate solution. 
\end{remark}

\begin{remark}
For $n<3$, regardless of the number of collected measurements, $\mathcal{M}[k]$ never becomes a singleton. This is due to the fact that Assumption~\ref{assum:independent} does not hold any more. 
\end{remark}

\begin{remark} The presented estimator also works if the step sizes are selected as $\alpha[k]=c/k^\delta$ for all $k\in\mathbb{Z}_{\geq 0}$ and for a fixed $\delta\in(1/2,1]$. This is true because we can always scale the measurements $(x[t],y[t])_{t=0}^{k}$ to $(x[t],t^\delta y[t])_{t=0}^{k}$ and, subsequently, use the presented results for the constant step size. If $\delta$ is not known by the eavesdropper, we can construct a filter to also reconstruct $\delta$ based on the measurements $(x[t],y[t])_{t=0}^{k}$, however, this would make the problem considerably more difficult because of the complexity of the Bayes updates in this case. 
\end{remark}

This remark shows that the choice of a time-varying step size as described above does not make the reconstruction of the utility function any harder than the constant step size case so long as $\delta$ is publicly known.

\subsection{Random Step Sizes Drawn from a Finite Set} \label{subsec:finite_stochastic_stepsize}
\ff{Next we consider the case where the step sizes $\alpha[k]$, for $k \in \mathbb{Z}_{\ge 0}$, are drawn uniformly from a set $\mathcal{A} = \{ \alpha^{(1)}, \dots, \alpha^{(s)}\}$ of $s$ distinct values.}
Moreover, we assume that the step sizes $(\alpha[k])_{k\in\mathbb{Z}_{\geq 0}}$ are independently and identically distributed over time. 
First, let us present a condition for which  Assumption~\ref{assum:independent} holds.

\begin{lemma} \label{lemma:independent:2} Let $\mathcal{Q}$ be such that the following condition is almost surely satisfied
\begin{align} \label{lemma:independent:2:condition}
\rank\left(
\begin{bmatrix}
x[0]^\top \\ x[0]^\top(I\ff{-}\alpha[0] Q)^\top \\ \vdots \\ x[0]^\top(I\ff{-}\alpha[0] Q)^\top\cdots (I\ff{-}\alpha[n-2] Q)^\top
\end{bmatrix}\right)=n.
\end{align}
Then, $(x[t])_{t=0}^{n-1}$ are almost surely independent.
\end{lemma} 
 
\begin{proof} Similar to the proof of Lemma~\ref{lemma:independent:1}, notice that $x[k+1]=(I\ff{-}\alpha[k] Q)x[k]\ff{-}\alpha[k] q.$ Therefore, we have
\begin{align} 
\begin{bmatrix}
x[0]^\top \\ x[1]^\top \\ \vdots \\ x[n-1]^\top
\end{bmatrix}
=&
\begin{bmatrix}
x[0]^\top \\ x[0]^\top(I\ff{-}\alpha[0] Q)^\top \\ \vdots \\ x[0]^\top(I\ff{-}\alpha[0] Q)^\top\cdots (I\ff{-}\alpha[n-2] Q)^\top
\end{bmatrix} \nonumber \\
&\ff{-}
\begin{bmatrix}
0_{1\times n} \\ (\alpha q)^\top \\ \vdots \\ \sum_{j=0}^{n-2}(\alpha[j] q)^\top(I\ff{-}\alpha[j+1] Q)^\top\cdots (I\ff{-}\alpha[n-2] Q)^\top
\end{bmatrix}.\label{eqn:rankequality:TV}
\end{align}
Since $x[0]$ is selected randomly and independently from the pair $(Q,q)$, the iterates $(x[t])_{t=0}^{n-1}$ are independent if the condition in the statement of the lemma holds.
\end{proof} 

\begin{remark} Condition~\eqref{lemma:independent:2:condition} is intimately related to the observability of time-varying linear systems \linebreak $((I\ff{-}\alpha[k] Q)^\top,x[0]^\top)$; see \cite[p.~462]{levine1996control}.
\end{remark}

At iteration $k+1$, after observing $(x[t],y[t])_{t=0}^k$, the eavesdropper may use the Bayes' update rule, consecutively, for generating the conditional density function
\begin{subequations}
\begin{align}
p(Q',q'|(x[t],y[t])_{t=0}^k) 
&\propto p(y[k]|Q',q',(x[t],y[t])_{t=0}^{k-1},x[k]) p(Q',q'|(x[t],y[t])_{t=0}^{k-1},x[k])\nonumber\\
&= p(y[k]|Q',q',(x[t],y[t])_{t=0}^{k-1},x[k]) p(Q',q'|(x[t],y[t])_{t=0}^{k-1}) \label{eqn:proof:equality1}\\
&= p(y[k]|Q',q',x[k]) p(Q',q'|(x[t],y[t])_{t=0}^{k-1})\label{eqn:proof:equality2}\\
&= \bigg[\sum_{\alpha'\in\mathcal{A}} p(y[k]|Q',q',x[k],\alpha')p(\alpha')\bigg] p(Q',q'|(x[t],y[t])_{t=0}^{k-1})\label{eqn:proof:equality3}\\
&\propto\bigg[\sum_{\alpha'\in\mathcal{A}} p(y[k]|Q',q',x[k],\alpha')\bigg] p(Q',q'|(x[t],y[t])_{t=0}^{k-1})\label{eqn:proof:equality4}
\end{align}
\end{subequations}
where~\eqref{eqn:proof:equality1} follows from independence of the cost parameters $(Q,q)$ and $x[k]$ given $(x[t],y[t])_{t=0}^{k-1}$ (note that $x[k]$ is merely \ff{the negation of} the summation of all $(y[t])_{t=0}^{k-1}$ \ff{plus} $x[0]$ and, hence, it is redundant information),~\eqref{eqn:proof:equality2} follows from independence of $y[k]$ from $(x[t],y[t])_{t=0}^{k-1}$ given $x[k]$ and the cost parameters $(Q,q)$,~\eqref{eqn:proof:equality3} follows from conditioning on $\alpha$, and~\eqref{eqn:proof:equality4} follows from the uniform distribution of the step sizes. Now, note that
\begin{align*}
\sum_{\alpha'\in\mathcal{A}} p(y[k]|Q',q',x[k],\alpha')
&=
\left\{
\begin{array}{ll}
1, & \exists \alpha''\in\mathcal{A}: y[k]=\alpha''(Q'x[k]+q'),\\
0, & \mbox{otherwise},
\end{array}
\right. \\
&=\mathbbm{1}_{(Q',q')\in \mathcal{D}(x[k],y[k])},
\end{align*}
where
$\mathcal{D}(x[k],y[k])=\big\{(Q',q')\in\mathcal{S}_{+}^n\times \mathbb{R}^n\,|\, \exists\alpha'\in\mathcal{A}: y[k]=\alpha'(Q'x[k]+q') \big\}.$ Hence, we have
$p(Q',q'| \; (x[t],y[t])_{t=0}^k)\propto \mathbbm{1}_{(Q',q')\in \mathcal{D}(x[k],y[k])}p(Q',q'|(x[t],y[t])_{t=0}^{k-1}).$
Using induction, we can show that
\begin{align*}
p(Q',q'|(x[t],y[t])_{t=0}^k)
&\propto \bigg[\prod_{t=0}^k\mathbbm{1}_{(Q',q')\in \mathcal{D}(x[t],y[t])}\bigg]p(Q',q')= \mathbbm{1}_{(Q',q')\in \cap_{t=0}^k\mathcal{D}(x[t],y[t])}p(Q',q').
\end{align*}
Now, we can redefine
\begin{align*}
\mathcal{M}[k]
&=\cap_{t=0}^k\mathcal{D}(x[t],y[t])=\big\{(Q',q')\in\mathbb{R}^{n\times n}\times \mathbb{R}^n\,|\,\exists \alpha'[t]\in\mathcal{A}: y[t]=\alpha'[t](Q'x[t]+q'), \forall t=0,\dots,k \big\}.
\end{align*}
Therefore, Bayes' rule gives
$p(Q',q'|(x[t],y[t])_{t=0}^k) = \mathbbm{1}_{(Q',q')\in \mathcal{M}[k]}p(Q',q').$ 
The next theorem shows that the set $\mathcal{M}[k]$ becomes a singleton and, hence, the Bayesian filter converges to the correct parameter selection.

\begin{theorem} \label{tho:finite_stochastic_stepsize} Let $n\geq 5$. Then $\mathcal{M}[k]=\{(Q,q)\}$ for all $k\geq \lceil (n+3)/2 \rceil$. 
\end{theorem}

\begin{proof} 
Let us enumerate all the possible sequences of the step sizes $\mathcal{A}^{k+1}$ for each $k$. Now, for any sequence of step sizes $(\alpha'[t])_{t=0}^{k}\in\mathcal{A}^{k+1}$, the consistent parameter sets are given by the set-valued mapping
$
\mathcal{Z}[k;(\alpha'[t])_{t=0}^{k}]=\{(Q'x[t]+q')=y[t]/\alpha'[t],\forall t=0,\dots,k\}.
$
Evidently, $\mathcal{Z}[k;(\alpha'[t])_{t=0}^{k}]$ is either a singleton or the empty set if 
$(n+1)/2\leq k$. This is the case because for the mentioned horizon length, given the sequences of the step sizes, the number of the equations (which we assumed are not redundant; see Assumption~\ref{assum:independent}) is larger than or equal the number of free variables. Now, if we get only one more measurement, i.e., $k=\lceil (n+3)/2 \rceil$, only one of these sets remain nonempty and that points to the true parameters. To be able to use Assumption~\ref{assum:independent}, we should have $\lceil (n+3)/2 \rceil=k\leq n-1$, which gives $n\geq 5$. 
\end{proof}

Note that the estimator constructed in the proof of Theorem~\ref{tho:finite_stochastic_stepsize} relies on the fact that we can enumerate all the possible sequences of the step sizes for $k=\lceil (n+3)/2 \rceil$ and solve a set of linear equations for each one to extract the true parameters. The number of all the possible sequences of the step sizes is equal to $s^{\lceil (n+3)/2 \rceil}=\mathcal{O}(s^n)$. Thus, this estimator is  practically implementable only for relatively small $s$ and $n$. \ff{However, this problem can be fixed with a simple change of variable. To do so, we can alternatively define the set of utility functions consistent with the observations as
\begin{align*}
\mathcal{M}[k]
=\big\{(Q',q')\in\mathbb{R}^{n\times n}\times \mathbb{R}^n\,|\,\exists \beta[t]\in\{1/\alpha^{(1)},\dots,1/\alpha^{(s)}\}: \beta[t]y[t]=Q'x[t]+q', \forall t=0,\dots,k \big\}.
\end{align*}
Following the same line of reasoning as in Remark~\ref{remark:explainNecessary}, we can see that the elements of $\mathcal{M}[k]$ are the solutions of the set of equations
\begin{align*} 
\left(
\begin{bmatrix}
x[0]^\top & 1 \\
\vdots & \vdots \\
x[k]^\top & 1
\end{bmatrix}
\otimes I
\right)
\begin{bmatrix}
\vec(Q') \\ q'
\end{bmatrix}
=
\begin{bmatrix}
y[0]\beta[0] \\
\vdots \\
y[k]\beta[k] 
\end{bmatrix}.
\end{align*}
We may rewrite this set of equations as
\begin{align*} 
\underbrace{
\begin{bmatrix}
x[0]^\top\otimes I & I & y[0]  & \cdots & 0 \\
\vdots & \vdots & \vdots & \ddots & \vdots \\
x[k]^\top\otimes I & I & 0 & \cdots &  y[k]
\end{bmatrix}
}_{\begin{array}{c}
:=\Phi
\end{array}}
\begin{bmatrix}
\vec(Q') \\ q' \\ \beta[0] \\ \vdots \\ \beta[k]
\end{bmatrix}
=
0.
\end{align*}
Using the arguments of Remark~\ref{remark:explainNecessary} and Lemma~\ref{lemma:independent:1}, we may observe that $\Phi$ is a full row rank matrix. Now, note that the number of unknown decision variables here is $k+n(n+1)/2+n$ and the number of equations is $(k+1)n$. Therefore, for this set of equations to have a unique solution (up to a scaling), we need to have $k+n(n+1)/2+n\leq (k+1)n$ which means
$k\geq n(n+1)/(2(n-1)).$ This is satisfied if we select $k\geq \lceil (n+3)/2 \rceil$ as recommended in Theorem~\ref{tho:finite_stochastic_stepsize}. Therefore, with this change of variable, we can reconstruct the utility function based on the observations in polynomial time.

\begin{remark}
Note that, in this subsection, we did not use the fact that $\mathcal{A}$ is a finite set. Therefore, the presented argument is also valid when the step sizes are selected from closed interval of the positive reals.
\end{remark}

\subsection{Agent-Dependent Step Sizes}
In this subsection, we assume that each entry of the decision variable $x[k]$ is updated by an agent that can select its step size independently. In this case, the gradient iterations for each entry of the decision variable becomes
\begin{align*}
x_i[k+1]=x_i[k]-\alpha_i[k]\bigg(q_{ii}x_i[k]+\sum_{j\neq i} q_{ij}x_j[k]+ q_i\bigg),
\end{align*}
where $\alpha_i[k]$ are independently and identically distributed discrete random  variables selected with equal probability from the set $\mathcal{A}$. Let $\alpha_i[k]$ and $\alpha_j[k]$ be statistically independent if $i\neq j$. As a result, we get
\begin{equation*}
x[k+1]=x[k]-A[k](Q x[k]+q),
\end{equation*}
where 
\begin{align*}
A[k]=
\begin{bmatrix}
\alpha_1[k] & \dots & 0 \\
\vdots & \ddots & \vdots \\
0 & \dots & \alpha_n[k]
\end{bmatrix}.
\end{align*}
We may define $\mathcal{A}^{n\times n}_d$ as the set of diagonal matrices of size $n\times n$ with entries belonging to $\mathcal{A}$. By definition, $A[k]\in\mathcal{A}^{n\times n}_d$ for all $k\geq 0$. Note that we can guarantee the convergence of the gradient algorithm even when using agent-dependent step sizes; see~\ref{appendix:convergence} for more information.
Similar to the previous subsections, let us define
\begin{align*}
y[k]=A[k](Qx[k]+q).
\end{align*}
As before, at time step $k + 1$, measurement pairs $(x[t],y[t])_{t=0}^k$ are available to the eavesdropper. Following the same line of reasoning as in the previous subsections, at iteration $k+1$, the eavesdropper may use the Bayes' update rule, consecutively, for generating the conditional density function
\begin{align}
p(Q',q'|(x[t],y[t])_{t=0}^k) 
&\propto
\mathbbm{1}_{(Q',q')\in \mathcal{D}(x[k],y[k])}
p(Q',q'|(x[t],y[t])_{t=0}^{k-1}).
\end{align}
where
$\mathcal{D}(x[k],y[k])=\big\{(Q',q')\in\mathcal{S}_{+}^n\times \mathbb{R}^n\,|\, \exists A'\in\mathcal{A}^{n\times n}_d: y[k]=A'(Q'x[k]+q') \big\}.$ Hence, using induction, we get
\begin{align*}
p(Q',q'|(x[t],y[t])_{t=0}^k)
&\propto \bigg[\prod_{t=0}^k\mathbbm{1}_{(Q',q')\in \mathcal{D}(x[t],y[t])}\bigg]p(Q',q')=\mathbbm{1}_{(Q',q')\in \mathcal{M}[k]}p(Q',q'),
\end{align*}
where
\begin{align*}
\mathcal{M}[k]
&=\cap_{t=0}^k\mathcal{D}(x[t],y[t])=\big\{(Q',q')\in\mathbb{R}^{n\times n}\times \mathbb{R}^n\,|\,\exists A'[t]\in\mathcal{A}^{n\times n}_d: y[t]=A'[t](Q'x[t]+q'), \forall t=0,\dots,k \big\}.
\end{align*}
The next theorem shows Bayesian filter is always inconclusive. 

\begin{theorem} \label{tho:notsingleton}
The cardinality of the set $\mathcal{M}[k]$ is uncountably infinite for all $k \in \mathbb{Z}_{\ge 0}$.
\end{theorem}

\begin{proof}
Firstly, note that we can redefine $\mathcal{M}[k]$ as 
\begin{align*}
\mathcal{M}[k]
&=\cap_{t=0}^k\mathcal{D}(x[t],y[t])=\big\{(Q',q')\in\mathbb{R}^{n\times n}\times \mathbb{R}^n\,|\,\exists B[t]\in\mathcal{B}^{n\times n}_d: B[k]y[t]=Q'x[t]+q', \forall t=0,\dots,k \big\},
\end{align*}
where $\mathcal{B}^{n\times n}_d$ denotes the set of all diagonal matrices of size $n\times n$ with diagonal entries belonging to $\mathcal{B}=\{1/\alpha^{(1)},\dots,1/\alpha^{(s)}\}$. This way, as described in the previous subsection, we need to solve a linear set of equations to find the entries of the set $\mathcal{M}[k]$. Now, note that, in each iteration, we receive $n$ new measurements while adding $n$ new variables  (i.e., step sizes). Therefore, even if all the measurements are independent, there is not a unique solution (since the number of unknowns is always strictly larger than the number of measurements). 
\end{proof}

Theorem~\ref{tho:notsingleton} shows that no matter how many measurements the eavesdropper gathers, it is impossible to reconstruct the cost function. 
}

\section{Constrained Case} \label{sec:qp_c}
For the constrained optimization problem, we add the constraints using logarithmic barrier functions. In that case, we get the unconstrained optimization problem
\begin{align*}
\max_{x\in\mathbb{R}^n} \quad - \frac{1}{2}  x^\top Q x-q^\top x + \lambda \sum_{i=1}^m \log(\ff{d}_i-\ff{C}_i x),
\end{align*}
where $\ff{C}_i\in\mathbb{R}^{1\times n}$ is $i$-th row of the matrix $\ff{C}$ and $\lambda>0$ is a scaling factor. As $\lambda$ approaches zero, the solution of this problem converges to the solution of the original constrained optimization problem~\cite[p.\,566]{boyd2004convex}. 
When using the gradient algorithm, in this case, we get
\begin{align*}
x[k+1]=x[k]\ff{-}\alpha[k]\left(Qx[k]+q\ff{+}\sum_{i=1}^m\frac{\lambda}{\ff{d}_i-\ff{C}_i x[k]}\ff{C}_i^\top\right), \, x[0]\in\mathcal{X}.
\end{align*}
Therefore, an eavesdropper that listens to consecutive iterations can construct the measurements
\begin{align*}
y[k]=\alpha[k]\left(Qx[k]+q\ff{+}\sum_{i=1}^m\frac{\lambda}{\ff{d}_i-\ff{C}_i x[k]}\ff{C}_i^\top\right).
\end{align*}
In the reminder of this section, we assume that the eavesdropper knows the parameters of the constraints $\ff{(C,d)}$ and only wants to estimate the parameters of the utility function and the scaling factor. 

\begin{remark} In many problems, the constraints are enforced by the physical characteristics of the problem and are hence public knowledge. This is different from the utility function that is typically motivated by the internal mechanism of the system and the priorities of its operator and is hence kept private.
\end{remark}

Let us consider the case that the step sizes are selected randomly and uniformly from a finite set $\mathcal{A}=\ff{\{\alpha^{(1)},\dots,\alpha^{(s)}\}}$ as in Subsection~\ref{subsec:finite_stochastic_stepsize}; the proofs for the other cases are not different. 

\begin{assumption} \label{assum:lambda} The parameter $\lambda\in\Lambda\subset\mathbb{R}_{\geq 0}$ is randomly generated according to the probability density function $p: \Lambda \rightarrow \mathbb{R}_{\geq 0}$. \ff{Assume that $x[0]$ is chosen uniformly at random from $\{x|Cx<d\}$.} Further, the distribution of $\lambda$ is independent of the initialization of the algorithm $x[0]$ and utility function parameters $(Q,q)$.
\end{assumption}

Similarly, we can present the following condition for the satisfaction of Assumption~\ref{assum:independent}.

\begin{lemma} \label{lemma:independent:3} Let $\mathcal{Q}$ be such that the following condition is almost surely satisfied
\begin{align*}
\rank\left(
\begin{bmatrix}
x[0]^\top \\ x[0]^\top(I\ff{-}\alpha[0] Q)^\top \\ \vdots \\ x[0]^\top(I\ff{-}\alpha[0] Q)^\top\cdots (I\ff{-}\alpha[n-2] Q)^\top
\end{bmatrix}\right)=n.
\end{align*}
Then, $(x[t])_{t=0}^{n-1}$ are almost surely independent.
\end{lemma} 
 
\begin{proof} The proof follows from the same line of reasoning as in Lemma~\ref{lemma:independent:2}. The only difference is that, in this case, the right-hand side of~\eqref{eqn:rankequality:TV} admits additional nonlinear terms that are multiplied by $\lambda$. However, since $\lambda$ is selected independently of the parameters and the initial condition (see Assumption~\ref{assum:lambda}), these terms almost surely do not contribute to the rank of the matrix on the left-hand side of~\eqref{eqn:rankequality:TV}.
\end{proof} 

At iteration $k+1$, after observing $(x[t],y[t])_{t=0}^k$, the eavesdropper may use the Bayes' update rule, consecutively, for generating the conditional density function
\begin{align*}
p(Q',q',\lambda'|(x[t],y[t])_{t=0}^k)
&= \bigg[\sum_{\alpha'\in\mathcal{A}} p(y[k]|Q',q',\lambda',x[k],\alpha')p(\alpha')\bigg] p(Q',q',\lambda',|(x[t],y[t])_{t=0}^{k-1})\\
&\propto\bigg[\sum_{\alpha'\in\mathcal{A}} p(y[k]|Q',q',\lambda',x[k],\alpha')\bigg]p(Q',q',\lambda'|(x[t],y[t])_{t=0}^{k-1}).
\end{align*}
Similarly, we have
$\sum_{\alpha'\in\mathcal{A}} p(y[k]|Q',q',\lambda',x[k],\alpha')=\mathbbm{1}_{(Q',q',\lambda')\in \mathcal{D}(x[k],y[k])},$
where
\begin{align*}
\mathcal{D}(x[k],y[k])=\bigg\{&(Q',q',\lambda')\in\mathcal{S}_{+}^n\times \mathbb{R}^n\times \Lambda\,|\,\exists\alpha'\in\mathcal{A}:  y[k]=\alpha'\bigg(Q'x[k]+q'\ff{+}\lambda'\sum_{i=1}^m\frac{1}{\ff{d}_i-\ff{C}_i x[k]}\ff{C}_i^\top\bigg) \bigg\}.
\end{align*}
Again, using induction, we can show that
\begin{align*}
p(Q',q',\lambda'|(x[t],y[t])_{t=0}^k)
&\propto \bigg[\prod_{t=0}^k\mathbbm{1}_{(Q',q',\lambda')\in \mathcal{D}(x[t],y[t])}\bigg]p(Q',q')p(\lambda')= \mathbbm{1}_{(Q',q',\lambda')\in \cap_{t=0}^k\mathcal{D}(x[t],y[t])}p(Q',q')p(\lambda').
\end{align*}
Now, we can define
\begin{align*}
\mathcal{M}[k]
&=\cap_{t=0}^k\mathcal{D}(x[t],y[t])\\
&=\bigg\{(Q',q',\lambda')\in\mathcal{S}_+^{n}\times \mathbb{R}^n\times \Lambda\,|\,\exists \alpha'[t]\in\mathcal{A}:  y[\ff{t}]=\alpha'\ff{[t]}\bigg(Q'x[\ff{t}]+q'\ff{+}\lambda'\sum_{i=1}^m\frac{1}{\ff{d}_i-\ff{C}_i x[\ff{t}]}\ff{C}_i^\top\bigg),\forall \ff{t=0,\dots,k} \bigg\},
\end{align*}
which gives 
$p(Q',q',\lambda'|(x[t],y[t])_{t=0}^k) = \mathbbm{1}_{(Q',q',\lambda')\in \mathcal{M}[k]}p(Q',q')p(\lambda').$ 
The next theorem shows that the Bayesian filter converges to the correct parameter selection.

\begin{theorem} \label{tho:constrained} Let $n\geq 6$. $\mathcal{M}[k]=\{(Q,q,\lambda)\}$ for all $k\geq \lceil (n+3)/2+1/n \rceil$.
\end{theorem}

\begin{proof} Let us assume that we use an estimator that, for each time step $k$, enumerates all the possible sequences of the step sizes $\mathcal{A}^{k+1}$. Now, for any sequence of step sizes $(\alpha'[t])_{t=0}^{k}\in\mathcal{A}^{k+1}$, the consistent parameter sets are given by the mapping
$\mathcal{Z}[k;(\alpha'[t])_{t=0}^{k}]=\{(Q'x[\ff{t}]+q'\ff{+}\sum_{i=1}^m\frac{\lambda'}{\ff{d}_i-\ff{C}_i x[k]}\ff{C}_i^\top)=y[t]/\alpha'[t],\forall t=0,\dots,k\}.$
Similar to the proof of Theorem~\ref{tho:finite_stochastic_stepsize}, $\mathcal{Z}[k;(\alpha'[t])_{t=0}^{k-1}]$ is either a singleton or the empty set if 
$k+1\geq ( (n+1)n/2 + n+1)/n=(n+3)/2 + 1/n.$
This is true since, given the sequences of the step sizes, the number of the equations is larger than or equal the number of free variables. Now, if we get one more measurement, i.e., $k=\lceil (n+3)/2+1/n \rceil$, only one of these sets remain nonempty. To be able to use Assumption~\ref{assum:independent}, we should have $\lceil (n+3)/2+1/n \rceil=k\leq n-1$, which gives $n\geq 6$. 
\end{proof}

\begin{remark} In the interior point method, the algorithm automatically shrinks the scaling factor $\lambda$ to extract the optimal point. If the rule for shrinking the scaling factor is known, one can use the Bayesian filter above to estimate the parameters of the utility function. 
\end{remark}

\ff{
Note that, similar to the previous section, we may introduce change of variables to reconstruct the set $\mathcal{M}[k]$, and subsequently the utility function, in polynomial time. To do so, note that we may alternatively define $\mathcal{M}[k]$ as
\begin{align*}
\mathcal{M}[k]
=\bigg\{(Q',q',\lambda')\in\mathcal{S}_+^{n}\times \mathbb{R}^n\times \Lambda\,|\,&\exists \beta[t]\in\{1/\alpha^{(1)},\dots,1/\alpha^{(s)}\}:\\ &  \beta[t]y[t]=Q'x[t]+q'+\lambda'\sum_{i=1}^m\frac{1}{d_i-C_i x[t]}C_i^\top,\forall t=0,\dots,k \bigg\}.
\end{align*}
This way, we get a set of linear equations, which as showed in Theorem~\ref{tho:constrained} admits a unique solution for $k\geq \lceil (n+3)/2+1/n \rceil$.

Alternatively, we can use independently selected random step sizes at each agent to render the problem of reconstructability impossible. 
In this case, the update gradient algorithm becomes
\begin{align*}
x[k+1]=x[k]-A[k]\left(Qx[k]+q+\sum_{i=1}^m\frac{\lambda}{d_i-C_i x[k]}C_i^\top\right), \, x[0]\in\mathcal{X},
\end{align*}
where $A[k]\in\mathcal{A}^{n\times n}_d$ contains the stochastically-varying agent-dependent step sizes. Therefore, an eavesdropper that listens to consecutive iterations can construct the measurements
\begin{align*}
y[k]=A[k]\left(Qx[k]+q+\sum_{i=1}^m\frac{\lambda}{d_i-C_i x[k]}C_i^\top\right).
\end{align*}
Now, employing the Bayesian filter, the estimator can deduce that
\begin{align*}
p(Q',q',\lambda'|(x[t],y[t])_{t=0}^k)
&\propto \mathbbm{1}_{(Q',q',\lambda')\in \mathcal{M}[k]}p(Q',q')p(\lambda'),
\end{align*}
where
\begin{align*}
\mathcal{M}[k]
=\bigg\{(Q',q',\lambda')\in\mathcal{S}_+^{n}\times \mathbb{R}^n\times \Lambda\,|\,\exists A[t]\in\mathcal{A}^{n\times n}_d:  y[t]=A[t]\bigg(Q'x[t]+q'+\lambda'\sum_{i=1}^m\frac{1}{d_i-C_i x[t]}C_i^\top\bigg),\forall t=0,\dots,k \bigg\}.
\end{align*} 
Now, we may prove the following impossibility results. 

\begin{theorem}
The cardinality of the set $\mathcal{M}[k]$ is uncountably infinite for all $k \in \mathbb{Z}_{\ge 0}$.
\end{theorem}

\begin{proof}
The proof is similar to that of Theorem~\ref{tho:notsingleton}.
\end{proof}

}

\section{Conclusions}\label{sec:conc}
In this paper, we studied the problem of how to keep a utility function confidential even when the network over which the function is being optimized is compromised. Particularly, we considered the problem where an eavesdropper's objective is to reconstruct a quadratic utility function via measuring the decision variable iterations under a gradient method. We considered the impact of different choices of the step size on the reconstructability of the utility function and showed that for the case that the step size is not constant and is selected randomly from a sufficiently large set of appropriate candidates, it is virtually impossible for an eavesdropper to reconstruct the utility function. \ff{Therefore, the best design recommendation is to add time-varying agent-dependent random step sizes to the implemented dynamics.}  In addition to time-varying random step sizes, there are other ingredients that matter, such as having a uniform random direction for the initial condition and not executing too many gradient descent steps (since if the number of steps is below a threshold the solution cannot be uniquely determined even if having access to extraordinary computational capabilities). An interesting avenue for future research could be to devise a tractable algorithm that can approximate the utility function  and bound the accuracy of the approximation based on the statistics of the step size selection method. This can be done by using set-membership identification techniques (see Remark~\ref{remark:smi}) for bounding the difference between the identified and the true set of permissible parameters. Another avenue for future research could be to also study quasi-Newton or Newton methods because they require fewer iterations for converge and, thus, potentially minimize the amount of the leaked information. 

\section*{References}
\bibliographystyle{elsarticle-num}
\bibliography{ref}

\appendix

\ff{
\section{Convergence of Gradient Algorithm with Agent-Dependent Step Sizes} \label{appendix:convergence}
To present the results of this appendix, we need to introduce some notation. For a matrix $A$, we write $A\leq 0$ if $A$ is a symmetric negative semi-definite matrix. Further, for matrices $A$ and $B$ of the same dimension, we write $A\leq B$ if $A-B\leq 0$. 

Note that we follow the update rule $x[k+1]=x[k]-A[k](Qx[k]+q)$, where $A[k]$ is the matrix of step sizes satisfying $c_1 I \leq A[k] \leq c_2 I$ for all $k$ with $0<c_1<c_2$. In the reminder of this appendix, we determine conditions on $c_1,c_2$ that guarantee the convergence of the gradient algorithm.  These iterates converge to the maximizer of the utility function so long as the they satisfy  Wolfe's conditions; see Theorem~3.2 in~\cite[p.\,38]{wright1999numerical}. From Wolfe's conditions, for $0<\epsilon_1<\epsilon_2<1$, we have
\begin{subequations}
\begin{align}
(x[k]- A[k](Qx[k]+q))^\top Q(x[k]- &A[k](Qx[k]+q))+q^\top[x[k]- A[k](Qx[k]+q)]-x[k]^\top Qx[k]-q^\top x[k]\nonumber\\
&\leq -\epsilon_1  (Qx[k]+q)^\top A[k](Qx[k]+q),
\label{eqn:remarkWolfe1}\\
(Qx[k]+q)^\top  A[k] (Q(x[k]-& A[k](Qx[k]+q))+q)\leq \epsilon_2(Qx[k]+q)^\top A[k] (Qx[k]+q).
\label{eqn:remarkWolfe2}
\end{align}
\end{subequations}
We may rewrite~\eqref{eqn:remarkWolfe1} and~\eqref{eqn:remarkWolfe2}, perspectively, as
\begin{align*}
 \begin{bmatrix}
A[k]^{\frac{1}{2}}Qx[k] \\ A[k]^{\frac{1}{2}}q
\end{bmatrix}^\top
\begin{bmatrix}
 A[k]^{\frac{1}{2}}QA[k]^{\frac{1}{2}}+(\epsilon_1-2)I& 
\dfrac{2\epsilon_1-3}{2}I+ A[k]^{\frac{1}{2}}QA[k]^{\frac{1}{2}}\\ 
\dfrac{2\epsilon_1-3}{2} I+  A[k]^{\frac{1}{2}}QA[k]^{\frac{1}{2}}&
 A[k]^{\frac{1}{2}}QA[k]^{\frac{1}{2}}+(\epsilon_1-1)I
\end{bmatrix}
\begin{bmatrix}
A[k]^{\frac{1}{2}}Qx[k] \\ A[k]^{\frac{1}{2}}q
\end{bmatrix}\leq 0,
\end{align*}
and
\begin{align*}
\begin{bmatrix}
A[k]^{\frac{1}{2}}Qx[k] \\ A[k]^{\frac{1}{2}}q
\end{bmatrix}^\top
\begin{bmatrix}
I \\ I
\end{bmatrix}
\bigg((1-\epsilon_2)I- A[k]^{\frac{1}{2}}QA[k]^{\frac{1}{2}}\bigg)
\begin{bmatrix}
I \\ I
\end{bmatrix}^\top
\begin{bmatrix}
A[k]^{\frac{1}{2}}Qx[k] \\ A[k]^{\frac{1}{2}}q
\end{bmatrix}\leq 0.
\end{align*}
These conditions are satisfied if the following inequalities hold
\begin{subequations}
\begin{align}
\begin{bmatrix}
(\epsilon_1-2)I & 
\dfrac{2\epsilon_1-3}{2} I \\ 
\dfrac{2\epsilon_1-3}{2} I &
(\epsilon_1-1)I 
\end{bmatrix}
+
 \begin{bmatrix}
I \\ I
\end{bmatrix}(A[k]^{\frac{1}{2}}QA[k]^{\frac{1}{2}})\begin{bmatrix}
I \\ I
\end{bmatrix}^\top\leq 0,
\label{eqn:appendixA:1}
\\
(1-\epsilon_2)I- A[k]^{\frac{1}{2}}QA[k]^{\frac{1}{2}}\leq 0.
\label{eqn:appendixA:2}
\end{align}
\end{subequations}
For~\eqref{eqn:appendixA:1} to hold, it is sufficient to satisfy
\begin{align} \label{eqn:appendixA:3}
\begin{bmatrix}
(\epsilon_1-2)I & 
\dfrac{2\epsilon_1-3}{2} I \\ 
\dfrac{2\epsilon_1-3}{2} I &
(\epsilon_1-1)I 
\end{bmatrix}
+
c_2\lambda_{\max}(Q) \begin{bmatrix}
I & I \\
I & I
\end{bmatrix}\leq 0,
\end{align}
where $\lambda_{\max}(Q)$ denotes the largest eigenvalue of $Q$. Using Schur's complement, we can translate~\eqref{eqn:appendixA:3} to
\begin{align*}
\epsilon_1-2+c_2\lambda_{\max}(Q) <0,\\
\bigg(\epsilon_1-1+c_2\lambda_{\max}(Q) \bigg)\bigg(\epsilon_1-2+c_2\lambda_{\max}(Q)  \bigg)-\bigg(\epsilon_1-\dfrac{3}{2}+c_2\lambda_{\max}(Q) \bigg)^2<0.
\end{align*}
This holds if $\epsilon_1-2+c_2\lambda_{\max}(Q) <0$. For~\eqref{eqn:appendixA:2} to hold, it is sufficient to have
$1-\epsilon_2- c_1\lambda_{\min}(Q)\leq 0$, where $\lambda_{\min}(Q)$ denotes the smallest eigenvalue of $Q$. Therefore, for the gradient algorithm to converge, it suffices to select $ c_1>(1-\epsilon_2)/\lambda_{\min}(Q)$ and $ c_2<(2-\epsilon_2)/\lambda_{\max}(Q)$ for $0<\epsilon_1<\epsilon_2<1$.
}

\end{document}